\documentstyle{amsart}
\def\R{{\hbox{\bf R}}}
\def\C{{\hbox{\bf C}}}

\def\P{{\hbox{\bf P}}}
\def\E{{\hbox{\bf E}}}

 at 10 true pt

\newenvironment{proof}{\noindent {\bf Proof} }{\endprf\par}
\def \endprf{\hfill  {\vrule height6pt width6pt depth0pt}\medskip}

\theoremstyle{plain}
\newtheorem{theorem}{Theorem}
\newtheorem{corollary}{Corollary}
\newtheorem{lemma}{Lemma}

\newtheorem{question}{Question}
\newtheorem{claim}{Claim}
\newtheorem{conjecture}{Conjecture}

\theoremstyle{remark}
\newtheorem{remark}{Remark}

\theoremstyle{definition}
\newtheorem{definition}{Definition}

\begin{document}
\title{Bilinear and quadratic variants on the Littlewood-Offord problem}

\author{Kevin P. Costello}
\address{Department of Mathematics, Georgia Institute of Technology, Atlanta, GA 30308}
\email{kcostell@@math.gatech.edu}
\thanks{This research was supported by NSF Grants DMS-0635607 and DMS-0456611}
\maketitle
\begin{abstract}
If $f(x_1, \dots, x_n)$ is a polynomial dependent on a large number of independent Bernoulli random variables, what can be said about the maximum concentration of $f$ on any single value?  For linear polynomials, this reduces to one version of the classical Littlewood-Offord problem: Given nonzero constants $a_1, \dots a_n$, what is the maximum number of sums of the form $\pm a_1 \pm a_2 \dots \pm a_n$ which take on any single value?  Here we consider the case where $f$ is either a bilinear form or a quadratic form.  For the bilinear case, we show that the only forms having concentration significantly larger than $n^{-1}$ are those which are in a certain sense very close to being degenerate.  For the quadratic case, we show that no form having many nonzero coefficients has concentration significantly larger than $n^{-1/2}$.  In both cases the results are nearly tight.  
\end{abstract}

\section{Introduction: The Linear Littlewood-Offord Problem}
In their study of the distribution of the number of real roots of random polynomials, Littlewood and Offord \cite{LO} encountered the following problem:
\begin{question}
 Let $a_1$, \dots $a_n$ be real numbers such that $|a_i|>1$ for every $i$.  What is the largest number of the $2^n$ sums of the form
\begin{equation*}
 \pm a_1 \pm a_2 \dots \pm a_n
\end{equation*}
  that can lie in any interval of length 1?
\end{question}
Littlewood and Offord showed an upper bound of $O(2^n \frac{\log n}{\sqrt{n}})$ on the number of such sums.  Erd\H{o}s \cite{Erd} later removed the $\log n$ factor from this result, giving an exact bound of $\binom{n}{\lfloor n/2  \rfloor}$ via Sperner's Lemma, which is tight in the case where all of the $a_i$ are equal.  The same bound was later shown by Kleitman \cite{Kle} in the case where the $a_i$ are complex numbers.  Rescaling Kleitman's result and using Sterling's approximation gives the following probabilistic variant of the lemma:
\begin{theorem} \label{thm:quantLO}
 Let $n>0$, and let $a_1, \dots a_n$ be arbitrary complex numbers, at least $m \geq 1$ of which are nonzero.  Let $x_1, \dots x_n$ be independent random variables drawn uniformly from $\{1, -1\}$.  Then
\begin{equation*}
 \sup_{c \in \R} \P(\sum_{i=1}^n a_i x_i=c) \leq \min\{\frac{1}{2}, \frac{1}{\sqrt{m}}\}
\end{equation*}
\end{theorem}
In a sense Theorem \ref{thm:quantLO} can be thought of as a quantitative description of the dispersion of a random walk: No matter what step sizes the walk takes, as the number of steps increases the walk becomes less and less concentrated on any particular value.  In this interpretation the $\sqrt{n}$ in the bound is also unsurprising; if the step sizes are small integers, we would expect the walk to typically be about at an integer about $O(\sqrt{n})$ distance from 0 at time $n$, so the concentration at individual points near 0 should be roughly $n^{-1/2}$.  

In 1977 Hal\'{a}sz \cite{Hal} gave several far reaching generalizations of Theorem \ref{thm:quantLO}, both to higher dimensions and to more general classes of random variables.  One (rescaled) result of his is
\begin{theorem} \label{thm:Halasz2d}
Let $a_1, \dots a_n$ be vectors in $\R^d$ such that no proper subspace of $\R^d$ contains more than $n-m$ of the $a_i$.  Let $x_1, \dots x_n$ be independent complex-valued random variables such that for some $\rho<1$, 
\begin{equation*}
\sup_{i, c} \P(x_i=c) \leq \rho.
\end{equation*}
Then
\begin{equation*}
 \sup_{c \in \R} \P(\sum_{i=1}^n a_i x_i =c) = O_{\rho, d}(m^{-d/2}).
\end{equation*}
\end{theorem}
The original Littlewood-Offord lemma corresponds to the special case where $d=1$ and the $x_i$ are iid Bernoulli variables.  Again this can be thought of as a dispersion result: a linear polynomial which depends on a large number of independent, moderately dispersed random variables will itself be very dispersed.  Furthermore, the dispersion will be greater if the coefficients of the polynomial are in some sense truly $d-$dimensional.  

One application of these results is in the study of random matrices, since several key parameters of a matrix (e.g. the determinant, or the distance from one row to the span of the remaining rows) are linear forms in the entry of a single row or column of the matrix.  Koml\'{o}s \cite{Kom} used Theorem \ref{thm:quantLO} in 1967 to show that a random Bernoulli matrix (one whose entries are independently either 1 or -1) is almost surely non-singular.  Later, Kahn, Koml\'{o}s and Szemer\'{e}di \cite{KKS} used the ideas of Hal\'{a}sz to show that the singularity probability was exponentially small of the size of the matrix.  The current best bound for this probability, $(\frac{1}{\sqrt{2}}+o(1))^n$ for an $n \times n$ matrix \cite{BVW}, comes from a detailed analysis of the \textit{inverse} of the Littlewood-Offord problem, which can be thought of as 
\begin{question}
 If $\sum a_i x_i$ is highly concentrated on one value, what can be said about the $a_i$?
\end{question}
The intuition here if the sum takes on a single value with probability close to $n^{-1/2}$, then the $a_i$ should be very highly structured.  Tao and Vu \cite{TV} and Rudelson and Vershynin \cite{RV} showed that this was in fact the case: If the sum takes on a single value with probability at least $n^{-c}$ for some fixed $c$, then the coefficients must have been drawn from a short generalized arithmetic progression.  One special case of this result can be expressed more quantitatively in the following theorem from \cite{TVquant}

\begin{theorem} \label{thm:sharpILO}
Let $a_1, \dots a_n$ be nonzero complex numbers, and let $\epsilon<\frac{1}{2}$ and $\alpha>0$ be fixed. Then there is an $N_0=N_0(\epsilon, \alpha)$ such that if $n>N_0$ and for $x_i$ independently and uniformly chosen from $\{1, -1\}$ 
\begin{equation*}
 \P(\sum_{i=1}^n a_i x_i = c) \geq n^{-1/2-\epsilon},
\end{equation*}
then there is a $d \in R$ such that  if $n>N_0$ all but $n^{1-\alpha}$ of the $a_i$ have the form
\begin{equation*}
 a_i= d b_i,
\end{equation*}
where the $b_i$ are integers such that $|b_i| \leq n^{\epsilon+\alpha}$.  

The same holds true if the $x_i$ are independent and identically distributed ``lazy walker'' variables satisfying  $\P(x_i=0)=2 \rho$, $\P(x_i=1)=\P(x_i=-1)=1-\rho$ for some $0<\rho<1$ ($N_0$ is now also dependent on $\rho$).  
\end{theorem}

\section{Statement of Main Results}
Our goal here will be to develop and strengthen extensions of Theorem \ref{thm:quantLO} and related results to polynomials of higher degree, in particular bilinear and quadratic forms.  To begin, let us consider the following result (implicit in \cite{CTV}), which we reprove here for convenience:
\begin{theorem} \label{thm:BiLO}
 Let $A=a_{ij}, 1 \leq i \leq m, 1 \leq j \leq n$ be an array of complex numbers, and suppose that at least $r$ distinct rows of $A$ each contain at least $r$ nonzero entries.  Let $x=(x_1, \dots x_m)$ and $y=(y_1, \dots y_n)$ be two vectors whose $m+n$ entries are random variables independently and uniformly chosen from $\{1, -1\}$.  Then
\begin{equation*}
 \sup_c \P(x^TAy=\sum_{i=1}^m \sum_{j=1}^n a_{ij} x_i y_j=c)=O(r^{-1/2})
\end{equation*}
\end{theorem}
\begin{proof}:
Without loss of generality we may assume that the rows in question correspond to the variables $x_1$ through $x_r$.  

Let $W_i=\sum_i a_{ij} y_j$, and let $W$ denote the number of $i$ between $1$ and $r$ for which $W_i$ is equal to 0.  We have
\begin{equation*}
 \P(x^T A y =c) \leq \P(W \geq \frac{r}{4})+\P(x^T A y=c \wedge W<\frac{r}{4}).
\end{equation*}
 We bound each term separately.  For the first term, we view $W$ as a sum of the indicator function of the events that each $W_i$ is equal to 0.  Since by Theorem \ref{thm:quantLO} each $W_i$ is equal to 0 with probability $O(r^{-1/2})$, it follows from linearity of expectation that $\E(W)=O(r^{1/2})$, and therefore from Markov's inequality that 
\begin{equation*}
 \P(W \geq \frac{r}{4}) =O(r^{-1/2}).
\end{equation*}
For the second term, we treat $y$ as fixed and write
\begin{equation*}
 x^TAy=\sum_i W_i x_i .  
\end{equation*}
If $W$ is at most $\frac{r}{4}$, then the right hand side is a linear form in the $x_i$ with at least $\frac{3r}{4}$ nonzero coefficients.  It follows from Theorem \ref{thm:quantLO} and taking expectations over $y$ that this term is $O(r^{-1/2})$.  
\end{proof}

In a certain sense this is a weaker result than we might expect.  If $A$ is an $n \times n$ matrix of small nonzero integers, then the magnitude of $x^T A y$ will typically be around $n$, so we might expect a concentration probability of $n^{-1}$ instead of $n^{-1/2}$.  However, Theorem \ref{thm:BiLO} is tight, as the polynomial $(x_1+...+x_n)(y_1+...+y_n)$ shows.  What our first main result shows is that every bilinear form with sufficiently large concentration probability is in some sense close to this degenerate example.
\begin{theorem} \label{thm:inversebi}
 Fix $\epsilon>0$.  Let $x=(x_1, \dots x_m)$ and $y=(y_1, \dots y_n)$ be independent random vectors whose entries are uniformly chosen from $\{1, -1\}$, and suppose that for some $r \leq m$ every row of the coefficient matrix $A$ of the bilinear form $x^TAy$ contains at least $r$ nonzero entries.  If $r$ and $m$ are sufficiently large and there is a function $f$ such that 
\begin{equation} \label{eqn:biconc}
 \P(x^T A y=f(y)) \geq r^{-1+\epsilon},
\end{equation}
then $A$ contains a rank one submatrix of size at least $(m-O_{\epsilon}(\frac{r}{\log^6 r})) \times (n-O_{\epsilon}(\frac{r}{\log^6 r})) $ (here the constant in the $O( )$ notation is as $r$ tends to infinity and is allowed to depend on $\epsilon$).

The same holds true if \eqref{eqn:biconc} holds when the entries of $y$ are independently set equal to 0 (with probability $1/2$) or $\pm 1$ (with probability $1/4$ each).  

In particular, this holds for the case where $f(y)=c$ is constant.  
\end{theorem}
\begin{remark}
 Note that we now require the stronger condition that \textit{every} row have many nonzero entries.  If this does not hold, we can first expose the $x_i$ corresponding to rows with few nonzero entries, then apply Theorem \ref{thm:inversebi} to the bilinear form on the remaining variables.  It follows that the rows of $A$ having many nonzero entries must correspond almost entirely to a rank one submatrix.  
\end{remark}
\begin{remark}
 The $-1$ in the exponent is sharp.  If $A$ is a small integer matrix, then $x^TAy$ will typically be on the order of $n$ in absolute value, so by the pigeonhole principle some value is taken on with probability $\Omega(n^{-1})$.  However, a randomly chosen such $A$ will with high probability not have rank one submatrices of size larger than $O(\log n)$.     
\end{remark}
In terms of the original bilinear form, a rank one submatrix corresponds to a form which factors completely as $x^TAy=g(x)h(y)$.   Theorem \ref{thm:inversebi} states that any bilinear form with sufficiently large concentration probability is highly structured in the sense that it can be made into one which factors by setting only a small portion of the variables equal to 0. 

We next turn our attention to quadratic forms $x^T A x$, where $x$ is again random.  Here we first aim to show
\begin{theorem} \label{thm:quadlo}
Let $A$ be an $n \times n$ symmetric matrix of complex numbers such that every row of $A$ has at least $r$ nonzero entries, where $r \geq exp((\ln n)^{1/4})$, and let $L$ be an arbitrary linear form.  Let $x$ be a vector of length $n$ with entries chosen uniformly and independently from $\{1, -1\}$.  Let $\delta>0$ be fixed.  Then
 \begin{equation*}
\sup_c  \P_A:=\P(x^TAx=L(x)+c)=O_{\delta} (r^{-1/2+\delta}).
 \end{equation*}
In particular, the above bound holds for the case where $L(x)$ is identically 0.
\end{theorem}
We will then remove the assumption that \textit{every} row of $A$ have many nonzero entries, obtaining the following corollary which may be easier to apply in practice
\begin{corollary} \label{cor:quadlo2}
 Let $A$ be an $n \times n$ symmetric matrix of complex numbers such that at least $mn$ of the entries of $A$ are nonzero, where $m \geq 3exp((\ln n)^{1/4})$.  Let $L$ and $x$ be above.  Then for any $\delta>0$,
\begin{equation*}
 \sup_c  \P_A:=\P(x^TAx=L(x)+c)=O_{\delta} (m^{-1/2+\delta}).
\end{equation*}
\end{corollary}
\begin{remark}
Again the $1/2$ is sharp, as can be seen from the form \\ $(x_1+ \dots +x_n)(x_1+\dots+x_m)$.  
\end{remark}

A weaker version of Theorem \ref{thm:quadlo} (with $\frac{1}{2}$ replaced by $\frac{1}{4}$)  was proved as a consequence of Theorem \ref{thm:BiLO} in \cite{CVrank}.  The improvement in the bound will come from a combination of Theorem \ref{thm:inversebi} and the use of a probabilistic variant of the Szemer\'{e}di-Trotter theorem.

We will prove Theorem \ref{thm:inversebi} in the next section, and the proof of Theorem \ref{thm:quadlo}  and Corollary \ref{cor:quadlo2} will come in the following section.  The remainder of the paper will be devoted to conjectured extensions of both results.

\section{The Proof of Theorem \ref{thm:inversebi}}
As in the proof of Theorem \ref{thm:BiLO}, we begin by dividing the vectors $y$ into two classes based on how many coordinates of $Ay$ are equal to 0.   

\begin{definition}
 A vector $y$ is \textbf{typical} if at least $r^{1-\frac{\epsilon}{4}}$ entries of $Ay$ are nonzero.  Otherwise it is atypical.
\end{definition}

Theorem \ref{thm:inversebi} is an immediate consequence of the following two lemmas.  
\begin{lemma}\label{lemma:typical}
 If $\P(x^TAy=f(y) \wedge y \textrm{ is typical }) \geq \frac{1}{2} r^{-1+\epsilon}$, then the conclusions of Theorem \ref{thm:inversebi} hold.  
\end{lemma}

\begin{lemma}\label{lemma:atypical}
If $\P(y \textrm{ is atypical }) \geq \frac{1}{2} r^{-1+\epsilon}$, then the conclusions of Theorem \ref{thm:inversebi} hold.
\end{lemma}
\begin{remark}
If we consider a form which factors perfectly as $x^TAy=g(x)h(y)$, then the hypothesis of Lemma \ref{lemma:typical} corresponds to the case where $g(x)$ is very structured (concentrated on a single value with probability close to $r^{-1/2}$), while that of Lemma \ref{lemma:typical} corresponds with the same property holding for $h(y)$.   
\end{remark}

We will examine each lemma in turn.  
\subsection{The proof of Lemma \ref{lemma:typical}}
We will assume throughout this section that $A$ is a matrix such that 
\begin{equation*}
 \P(x^TAy=f(y) \wedge y \textrm{ is typical }) \geq \frac{1}{2} r^{-1+\epsilon}
\end{equation*}

It follows from Lemma \ref{thm:quantLO} that for any $y_0$ which is typical we have 
\begin{equation} \label{eqn:LOtypical}
 \P_x(x^TAy_0=f(y_0)) \leq r^{-\frac{1}{2}+\frac{\epsilon}{8}}.  
\end{equation}

Our argument will go roughly as follows: Under our assumptions, we know that there must be many typical $y_0$ for which \eqref{eqn:LOtypical} is not too far from equality.  By Theorem \ref{thm:sharpILO}, we know that for such $y_0$ the coordinates of $Ay_0$ must be very highly structured, in the sense that all of them except for a small exceptional set must lie in not too long an arithmetic progression.  

The difficulty is that the exceptional sets in Theorem \ref{thm:sharpILO} may be different for different $y_0$. However, there will still be many ``small'' (of size much smaller than $n$) sets of coordinates which will lie entirely outside the exceptional set for most $y$.  We will show that such sets correspond to small collections of rows in $A$ which are very close to being multiples of each other, and then aggregate those collections to find our $A'$.  We now turn to the details.    

We will make use of the following (truncated) quantitative description of how embeddable a small group of real numbers is in a short arithmetic progression, which can be thought of as a variant of the essential LCD used in \cite{RV}.  

\begin{definition}
The \textbf{commensurability} of a $k-$tuple $(a_1, \dots a_k)$ of real numbers is defined by 
\begin{equation*}
 Comm(a_1, \dots a_k)=\max\{r^{-\frac{1}{2}+\frac{\epsilon}{4}}, \frac{1}{R}\},
\end{equation*}
where $R$ is the length of the shortest arithmetic progression containing $0$ and every $a_i$ simultaneously.   
\end{definition}

For example, if $a \leq b$ are positive integers, then, up to the truncation at $r^{-\frac{1}{2}+\frac{\epsilon}{4}}$, $Comm(a,b)=\frac{b}{GCD(a,b)}$.  Also, if $(a_1, a_2, \dots a_k)$ are all drawn from an arithmetic progression of length $q$ containing 0, we are trivially guaranteed that $Comm(a_1, \dots , a_k)$ is at least $\frac{1}{q}$.  We next characterize the ``small sets'' of coordinates mentioned above in terms of this commensurability.   

\begin{definition}
A $k-$tuple $(v_1, v_2, \dots v_k)$ of vectors is \textbf{neighborly} if 
\begin{equation*}
 \E_y Comm(v_1^T y, v_2^T y, \dots v_k^T y) \geq \frac{1}{6} r^{-\frac{1}{2}+\frac{5 \epsilon}{8}}
\end{equation*} 
\end{definition}

Fix $k_0:=\log^7 r$.  Our next lemma states that the number of neighborly tuples is quite large:   
\begin{lemma} \label{lemma:manyneighborly}
 For $k \leq k_0$, there are at least $m^k(1-\frac{r^{1-\frac{\epsilon}{8}}}{m})$ neighborly $k-$tuples such that each $v_i^T$ is a row of $A$.  
\end{lemma}

The proof of this lemma will be deferred to a later section.  Our next goal will be translate the neighborliness of a tuple into structural information about the corresponding rows of $A$.  One natural way in which a tuple can be neighborly is if the rows in $A$ are themselves small multiples of each other, in which case the corresponding coordinates of $Ay$ will \textit{always} be small multiples of each other.  Our next lemma states that every neighborly tuple is in some sense close to this example.  
\begin{lemma} \label{lemma:neighborstruc}
 Let $k \leq k_0$, and let $(v_1, v_2, \dots v_k)$ be neighborly.  Then there are unique real numbers $d_2, \dots d_k$ and sets $S_2, \dots S_k$ of coordinates such that
\begin{itemize}
 \item {For each $j$, $v_1=d_j v_j$ on all coordinates outside of $S_j$}
 \item {$\prod_{j=2}^k |S_j \backslash \bigcup_{i=2}^{j-1} S_i|_1 =O_{\epsilon} (r^{1-\frac{5\epsilon}{4}})$,
where $|S|_1=1$ if $S$ is empty and $|S|_1=\min\{|S|,4\}$ otherwise.}
\end{itemize}
\end{lemma}
What's important here is that not only does each row differ only in a few places from being a multiple of the first row in the tuple (the exceptional sets are of size  $o(r)$), but also that the exceptions will tend to occur in the same columns.  This latter fact will help keep the exceptional sets from growing too quickly when we attempt to examine many neighborly tuples at once.  Again we will defer the proof of this lemma to a later section.  

 Together, the above two lemmas state that the matrix $A$ must have a great deal of local structure, in the sense that many not-too-large collections of rows are very close to being multiples of each other.  Our goal will now be to combine these into a single global structure.  Using Lemmas \ref{lemma:manyneighborly} and \ref{lemma:neighborstruc}, we will be able to prove the following weakened version of Theorem \ref{thm:inversebi}, which allows the number of exceptional rows to be proportional to $m$ instead of $r$.  
\begin{lemma} \label{lemma:largesubmatrix}
 If $A$ satisfies the hypotheses of Theorem \ref{thm:inversebi}, then $A$ contains a rank one submatrix of size $(m-O_\epsilon(\frac{m}{\log^6 r})) \times (n-O_\epsilon(r^{1-\frac{5 \epsilon}{4}}))$.  
\end{lemma}

In the following sections we will first prove Lemma \ref{lemma:largesubmatrix} assuming the truth of Lemmas \ref{lemma:manyneighborly} and \ref{lemma:neighborstruc}, then leverage that result into the stronger bound required by Theorem \ref{thm:inversebi}.  We will finish the proof of Lemma \ref{lemma:typical} by proving Lemmas \ref{lemma:manyneighborly} and \ref{lemma:neighborstruc}.

\subsection{The proof of Lemma \ref{lemma:largesubmatrix} assuming lemmas \ref{lemma:manyneighborly} and \ref{lemma:neighborstruc}}
Motivated by the conclusion of Lemma \ref{lemma:neighborstruc}, we make the following definition:
\begin{definition}
 Let $V=\{v_1, \dots v_k\}$ be an (ordered) neighborly $k-$tuple.  The \textbf{score} of $V$ is given by 
\begin{equation*}
 Score(V)=\sum_{j=2}^k \chi(S_j \notin \bigcup_{i=2}^{j-1} S_i),
\end{equation*}
 where the $S_j$ are as in Lemma \ref{lemma:neighborstruc} and $\chi(E)$ is the indicator function of the event $E$.    
\end{definition}
The score is well defined, since the $d_j$ and $S_j$ are unique in that lemma.  It also has the following useful properties
\begin{itemize}
 \item{$Score(v_1 \dots v_k) \leq Score(v_1 \dots v_{k+1})$.  Equality holds iff $S_{k+1} \subseteq \bigcup_{i=1}^k S_i$.}
 \item{If $(v_1, \dots v_k)$ is neighborly, then there can be at most $\\ \log_4 (O_{\epsilon}( r^{1-\frac{5\epsilon}{4}})) < \log r-1$ different $j$ for which the score increases from $(v_1, \dots v_j)$ to $(v_1, \dots v_{j+1})$.}   
\end{itemize}

For a given (ordered) neighborly $k$-tuple $V=(v_1, \dots v_k)$ of rows of $A$ with $k<k_0$, let $S(V)$ be the collection of all rows $v$ of $A$ such that $(v_1, \dots v_k, v)$ is a neighborly tuple with the same score as $V$.  Note that for any $V$, all of the rows in $S(V)$ are multiples of $v_1$ (and thus of each other) except in the coordinates  where a prior $d_j v_j$ differed from $v_1$, and the number of such coordinates is at most
\begin{equation*}
 |\bigcup_{j=2}^k S_j| = \sum_{j=2}^k |S_j \backslash \bigcup_{i=2}^{j-1} S_i| =O_{\epsilon}( r^{1-\frac{5\epsilon}{4}})
\end{equation*}
 by Lemma \ref{lemma:neighborstruc}.  It follows that we have a rank one submatrix of dimensions $\\ |S(V)| \times n-O_\epsilon(r^{1-\frac{5\epsilon}{4}})$.  It therefore suffices to show some $S(V)$ is large.  Let $b$ be the maximal value of $|S(V)|$ over all neighborly tuples of size at most $k_0-1$.  We count the number of neighborly $k_0-$tuples in two ways.

\textbf{Method 1:} By Lemma \ref{lemma:manyneighborly}, there are at least $m^{k_0}(1-\frac{r^{1-\frac{\epsilon}{8}}}{m})$ such tuples.

\textbf{Method 2:} We can bound the number of such tuples by first choosing a set $J$ of size $\log r-1$ of places in which the score is allowed to increase, then restricting our attention only to those tuples whose scores increase only on $J$.  For each $j$ where the score fails to increase from $(v_1, \dots v_j)$ to $(v_1 \dots v_{j+1})$, there are at most $b$ choices for $v_{j+1}$.  For each other $j$, there are at most $m$ choices.  It follows that the number of tuples is at most 
\begin{equation*}
 \binom{k_0-1}{\log r-1} m^{\log r} b^{k_0 - \log r} \leq k_0^{\log r} m^{\log r}b^{k_0-\log r}.
\end{equation*}

Comparing our methods, we have
\begin{equation*}
 1-\frac{r^{1-\frac{\epsilon}{8}}}{m} \leq \left(\frac{b}{m}\right)^{k_0-\log r} k_0^{\log r}.  
\end{equation*}
Using the relationship $e^{(-1+o_x(1))x} \leq 1-x \leq e^{-x}$, we have
\begin{equation*}
 e^{-(1+o(1))\frac{r^{1-\frac{\epsilon}{8}}}{m}} \leq e^{-(k_0-\log r)\frac{m-b}{m}+\log r \log k_0}
\end{equation*}
 Taking logs and using the definition of $k_0$ gives
\begin{equation*}
 \frac{m-b}{m} \leq \frac{\log m \log k_0 + (1+o(1)) \frac{r^{1-\frac{\epsilon}{8}}}{m}}{k_0 - \log r} =O(\frac{1}{\log^6 r}).
\end{equation*}
It follows that $b \geq m-O(\frac{m}{log^6 r})$, so we are done.  

\subsection{The proof of Lemma \ref{lemma:typical}, from Lemma \ref{lemma:largesubmatrix}}
We construct our rank one submatrix using the following procedure.  Let $A_0$ be a rank one submatrix of $A$ of size $(m-O(\frac{m}{log^6 r})) \times (n-O(r^{1-\frac{5\epsilon}{4}}))$ (such a matrix is guaranteed to exist by Lemma \ref{lemma:largesubmatrix}).  We initialize $X_1 \subseteq \{x_1, \dots x_n\}$ to be the variables corresponding to the rows of $A_0$, and $X_2$ to be the remaining variables, and $X_3$ to initially be empty.  We also initially set 
$Y_1$ to be the variables corresponding to the columns of $A_0$.  We now repeatedly follow the following procedure:

If the matrix corresponding to $(X_1 \cup X_2) \times Y_1$ has rank one, stop.  If this is not the case, choose $x_i \in X_1, x_j \in X_2,$ and $y_k, y_l \in Y_1$ such that $a_{ik} a_{jl} \neq a_{il}a_{kj}$.  Move $x_j$ from $X_2$ to $X_3$, and remove $y_k$ and $y_l$ from $Y_1$.

We can always find the necessary $x_i$ and $x_j$ since the matrix on $X_1 \times Y_1$ will always be a rank one matrix due to our choice of $A_0$.  It remains to check that this procedure in fact terminates after at most $O(\frac{r}{\log^5 r})$ steps, so that the final rank one matrix is sufficiently large.   Let us assume to the contrary that this does not occur.  

Let $S$ be a set of size $r$ formed by taking $\frac{r}{\log^6 r}$ variables from $X_3$ and $r-\frac{r}{\log^6 r}$ variables from $X_1$, and let $T$ be the remaining variables in $X$.  Let $\widetilde{A}$ be the submatrix of $A$ consisting of the rows corresponding to $S$.  We can write
\begin{equation*}
 x^TAy-f(y)=x_S^T\widetilde{A}y-g(y, x_T),
\end{equation*}
where $x_S$ (resp. $x_T$) is the vector of variables in $S$ (resp. $T$) By assumption we have 
\begin{eqnarray*}
 r^{-\frac{1}{2}-\epsilon} &\leq& \P(x^TAy=f(y)) \\
&=&\E_T (\P_S (x_S^T \widetilde{A}y=g(y, x_T)) \\
&\leq& \sup_{x_T} \P(x_S^T \widetilde{A} y = g(y, x_T)). 
\end{eqnarray*}
It follows from Lemma \ref{lemma:largesubmatrix} that $\widetilde{A}$ must contain a rank one submatrix of size $\\ (r-O(\frac{r}{\log^6 r})) \times (n-O(r^{1-\frac{5\epsilon}{4}}))$.   Since the number of excluded variables is much smaller than $\frac{r}{\log^5 r}$, there must be a variable $x_j \in X_3$ such that both $x_j$ and the corresponding $y_k$ and $y_l$ are contained in this submatrix, as well as some variable $x_{i'} \in X_1$.  However, this is a contradiction, as $a_{i'k} a_{jl} \neq a_{i'l}a_{kj}$.

\subsection{The proof of Lemma \ref{lemma:manyneighborly}}
We define $g_y$ and $D_y$  as follows:
\begin{itemize}
 \item{If $y$ is atypical, then $g_y=0$ and $D_y=\{1, \dots m\}$}
 \item{If $y$ is typical and no arithmetic progression of length at most $r^{\frac{1}{2}-\frac{\epsilon}{4}}$ contains at least $m- r^{1-\frac{\epsilon}{4}} $ of the elements of $Ay$, then $g_y=r^{-\frac{1}{2}+\frac{\epsilon}{4}}$ and $D_y=\{1, \dots m\}$} 
 \item{Otherwise, let $R$ be an arithmetic progression of minimal length containing 0 and at least $m- r^{1-\frac{\epsilon}{4}}$ elements of $Ay$.  We define $g_y=|R|^{-1}, $ and $D_y$ to be those $i$ such that the $i^{th}$ coordinate of $Ay$ is in $R$.}
\end{itemize}
Note that in this definition the $D_y$ are not uniquely determined.  We choose one arbitrarily for each $y$.   Furthermore, by construction, for any $k-$tuple contained in $D_y$, we have $Comm(a_1, \dots a_k) \geq g_y$.  

By viewing the Inverse Littlewood-Offord Theorem \ref{thm:sharpILO} in the ``forward'' direction we can now obtain the following:
\begin{lemma} \label{lemma:reverseinverse}
For every fixed $\epsilon<\frac{1}{2}$ there is an $r_0>0$ such that for all matrices $A$ with $r>r_0$ and all typical $y^*$ we have  
\begin{equation*}
 \P(x^TAy= f(y)| y=y^*) \leq r^{-\frac{1}{2}+\frac{3\epsilon}{8}} g_{y^*}.
\end{equation*}
\end{lemma}
\begin{proof} (of Lemma \ref{lemma:reverseinverse}):  
Since by construction $g_{y^*} \geq r^{-\frac{1}{2}+\frac{\epsilon}{4}}$, there is nothing to prove unless the probability in question is at least $r^{-1+\frac{5\epsilon}{8}}$, which we will assume to be the case.    
Let $r_1$ be the number of nonzero coefficients of $x^TAy^*$, viewed as a linear form in $x$, and let $\P(x^TAy^*=f(y^*))=r_1^{-\frac{1}{2}-\epsilon_0}$.  Since $y^*$ is typical, $r_1 \geq r^{1-\frac{\epsilon}{4}}$. In particular, this implies that $\epsilon_0<\frac{1}{2}$.   

 Applying Theorem \ref{thm:sharpILO} to this form with $\alpha=\frac{\epsilon}{4}$, we see there is an arithmetic progression containing all but $r_1^{1-\frac{\epsilon}{4}}$ coefficients and of length 
\begin{eqnarray*}
 r_1^{\epsilon_0+\alpha}&=&\frac{r_1^{-\frac{1}{2}+\frac{\epsilon}{4}}}{\P(x^TAy^*=f(y^*))} \\
&\leq& \frac{r^{(-\frac{1}{2}+\frac{\epsilon}{4})(1-\frac{\epsilon}{4})}}{\P(x^TAy^*=f(y^*))}  \\
&=& \frac{r^{-\frac{1}{2}+\frac{3\epsilon}{8}}}{\P(x^TAy^*=f(y^*))} r^{-\frac{\epsilon^2}{16}}.
\end{eqnarray*}
If follows that $g_{y^*} \geq r^{\frac{1}{2}-\frac{3 \epsilon}{8}} \P(x^TAy^*=f(y^*))$ as desired.

\end{proof}

Taking expectations over all $y$, we see that
\begin{equation*}
\P(x^T A y = f(y) \wedge y \textrm{ is typical } ) \leq r^{-\frac{1}{2}+\frac{3\epsilon}{8}} \E_y(g_y),
\end{equation*}
which combined with the hypothesis of Lemma \ref{lemma:typical} in turn implies that
\begin{equation} \label{eqn:gexpect}
 \E_y(g_y) \geq r^{-\frac{1}{2}+\frac{5\epsilon}{8}}
\end{equation}

Let $Z$ be the collection of $k-$tuples satisfying
\begin{equation*} 
 \E_y(g(y) \chi (\{a_1, \dots a_k\} \subseteq D_y)) \geq \frac{1}{3} \E_y (g_y)
\end{equation*}
By \eqref{eqn:gexpect}, every tuple in $Z$ is neighborly.  It remains to check that $|Z|$ is large.  

Since by construction $|D_y| \geq m-r^{1-\frac{\epsilon}{4}}$ for every $y$, we have 
\begin{eqnarray*}
 \E_{a_1, \dots a_k} \E_y(g_y \chi (\{a_1, \dots a_k\} \subseteq D_y) ) &=& \E_y(g_y \P(\{a_1, \dots a_k\} \in D_y)) \\
&\geq& (\frac{m-r^{1-\frac{\epsilon}{4}}}{m})^k \E_y(g_y).
\end{eqnarray*}
Combining this with the definition of $Z$, we have 
\begin{equation*}
 |Z|\E_y (g_y) +\frac{\E_y(g_y)}{3}(m^k-|Z|) \geq (m-r^{1-\frac{\epsilon}{4}})^k \E_y(g_y)  \geq m^k(1-\frac{kr^{1-\frac{\epsilon}{4}}}{m}) \E_y(g_y)
\end{equation*}
Solving the above inequality, we obtain
\begin{equation*}
 |Z| \geq m^k(1-\frac{3k}{2} \frac{r^{1-\frac{\epsilon}{4}}}{m}) \geq m^k(1-\frac{r^{-\frac{\epsilon}{8}}}{m})
\end{equation*}
and we are done.

\subsection{The proof of Lemma \ref{lemma:neighborstruc} for $k=2$}.  
Let $(a,b)$ be a pair of neighborly vectors.  Our goal will be to show that they are very close to being multiples of each other.   

We make use of the general fact that for any random variable $X$ taking values between 0 and 1
\begin{equation} \label{eqn:expectationbound}
 \E( X ) = \int_{u=0}^1 \P(X>u) du = \int_{t=1}^{\infty} \frac{\P(X>\frac{1}{t})}{t^2} dt
\end{equation}

In our case $X$ will be $Comm(a^T y,b^T y)$, so bounding the right hand side becomes a question of how likely it is for $a^T y$ and $b^T y$ to be embeddable in a progression of a given length.  We make the following further definitions:
\begin{definition}
 A pair $(l_1,l_2)$ of integers is \textbf{degenerate} for the vector pair $(a,b)$ if $l_1a$ and $l_2b$ agree in at least $n-\frac{r}{5}$ positions and at least one of $l_1$ and $l_2$ is nonzero.  
\end{definition}
Note that there is (up to multiples) at most one degenerate pair for $(a,b)$.  

We further define for an integer $q$ 
\begin{equation*}
 p_{ab}(q):= \P( \exists (l_1,l_2) \neq (0,0) | (l_1,l_2) \textrm{ is non-degenerate} \wedge l_1 a^T y = l_2 b^T y \wedge |l_1|, |l_2| \leq q)
\end{equation*}
Using these definitions and the definition of $Comm(a,b)$, we have
\begin{eqnarray*} 
 r^{-\frac{1}{2}+\frac{5\epsilon}{8}} &\leq& \E_y(Comm(a^Ty, b^Ty)) \\
&\leq& (\int_{t=1}^{r^{\frac{1}{2}-\frac{\epsilon}{4}}} \frac{p_{ab}(q)}{q^2} dq)+r^{-\frac{1}{2}+\frac{\epsilon}{4}}+ \P(k_0 a^Ty=l_0 b^T y \textrm{ for a degenerate }(k_0, l_0)),
\end{eqnarray*}

The middle term on the right hand side is negligible, and we will next show that the first term is also small by showing 
\begin{lemma} \label{lemma:doubleembed}
 For any positive $\alpha>0$, any $q<\sqrt{r}$ and any $a$ and $b$, there is a constant $C_\alpha$ dependent only on $\alpha$ such that $p_{ab}(q) \leq \frac{C_\alpha q}{r^{1/2-\alpha}}$.   
\end{lemma}
We may without loss of generality assume $C_\alpha>1$.  It follows that for any $0<\alpha<\frac{1}{2}$, assuming Lemma \ref{lemma:doubleembed}, we have  
\begin{eqnarray*}
 \int_{t=1}^{\infty} \frac{p_{ab}(q)}{q^2} dq  &\leq& \int_{1}^{r^{1/2-\alpha}} \frac{C_\alpha}{q r^{1/2-\alpha}} dq + \int_{r^{1/2-\alpha}}^\infty \frac{dq}{q^2} \\
&=&O_{\alpha} (r^{-1/2+\alpha} \ln r)
\end{eqnarray*}
By taking $\alpha$ sufficiently close to $0$, we see that for large $r$ the contribution from the first term is also $o(r^{-\frac{1}{2}+\frac{5\epsilon}{8}})$.

It follows that the dominant contribution to the expectation must come from the third term.  This implies that a degenerate pair $(k_0, l_0)$ exists, and that we furthermore must have.
\begin{equation*}
 \P(k_0 a^T y = l_0 b^T y)  \geq \frac{1}{12} r^{-\frac{1}{2}+\frac{5\epsilon}{8}}.  
\end{equation*}
It follows by the linear Littlewood-Offord lemma that the linear form $(k_0 a^T-l_0 b^T)y$ must have $O(r^{1-\frac{5\epsilon}{4}})$ nonzero coefficients.  But this is exactly what the lemma requires.  

\subsection{The proof of Lemma \ref{lemma:doubleembed}}
It suffices to prove the following:

\begin{lemma}
 Let $a_1, \dots a_n$ and $b_1, \dots b_n$ be fixed (real or complex) constants such that for each $i$ at least one of $a_i$ and $b_i$ is non-zero.  Let $x_1, \dots x_n$ be iid variables uniformly chosen from $\{-1, 1\}$.  Let $E_q$ be the event that there exist $u$ and $v$ satisfying
\begin{itemize}
 \item{$|u|, |v| \leq q$}
 \item{There are at least $\frac{n}{10}$ different $i$ for which $va_i \neq u a_i$.}
 \item{$v \sum a_i x_i=u\sum b_i x_i$}
\end{itemize}
Then for any $\alpha>0$ and any $1 \leq q<\sqrt{n}$,
\begin{equation*}
 \P(E_q) = O(\frac{q}{\sqrt{n}}n^\alpha),
\end{equation*}
where the constant implicit in the $O$ notation is as $n$ tends to infinity and may depend on $\alpha$.  
\end{lemma}
We will throughout assume that both $q$ and $n$ are tending to infinity.  By utilizing a Freiman isomorphism of order $2n^2$ (see for example \cite{TVadd}, Lemma 5.25), we may assume that the $a_i$ and the $b_i$ are all real integers.  We may furthermore without loss of generality assume for every $i$ either $b_i$ is positive or $b_i=0$ and $a_i$ is positive.  

Let $k$ be a positive integer satisfying that $k>\frac{1}{\alpha}$. We define $L_0=1$ and for $1 \leq j \leq k$, we define
\begin{equation*}
 L_j=\sup_{(c,d) \in \C^2} |\{(i_1, \dots i_j) : a_{i_1}+\dots+a_{i_j}=c \wedge b_{i_1}+\dots+b_{i_j}=d\}|
\end{equation*}
Clearly $1 \leq L_j \leq n^j$, and by treating $i_j$ as fixed we furthermore see that $ \\L_{j-1} \leq L_j \leq nL_{j-1}$.  This implies that one of the following two cases must hold
\begin{itemize}
\item{There is a $j$ between $1$ and $k$ for which $L_j \geq \frac{n}{q^{\frac{2}{2k+1}}}L_{j-1}$}
\item{$L_k \leq \frac{n^k}{q^{\frac{2k}{2k+1}}}$}
\end{itemize}
We handle each case separately.

\textbf{Case 1: $L_k \leq n^k q^{-\frac{2k}{2k+1}}$}.  Here we will make use of the following result of Hal\'{a}sz (implicit in \cite{Hal}, see also \cite{TVadd}):
\begin{theorem}
 Let $k>0$ be fixed, and let $v_1, \dots v_n$ be nonzero (real or complex) coefficients.  Let $R_k$ be the number of $2k-$tuples $(i_1, \dots i_k, j_1, \dots j_k)$ for which $a_{i_1}+\dots+a_{i_k}=a_{j_1}+\dots+a_{j_k}$.  Then for $x_i$ uniformly chosen from $\{-1, 1\}$,
\begin{equation*}
 \P(\sum_{i=1}^n a_i x_i=c) =O(n^{-2k-1/2}R_k).  
\end{equation*}
\end{theorem}

Combining the above result and the union bound, we can write
\begin{eqnarray*}
& & \P(E_q) \leq \sum_{(u,v)} \P(\sum_{i=1}^n (a_i v - b_i u)x_i=0) \\
&=& O(n^{-2k-1/2}) \sum_{(u,v)} \sum_{(i_1, \dots i_k) \atop (j_1, \dots j_k)} \chi(v(a_{i_1}+\dots+a_{i_k}-a_{j_1}-\dots-a_{j_k})=-u(b_{i_1}+\dots+b_{i_k}-b_{j_1}-\dots-b_{j_k})) \\
&=&O(n^{-2k-1/2})  \sum_{(i_1, \dots i_k) \atop (j_1, \dots j_k)} \sum_{(u,v)} \chi(v(a_{i_1}+\dots+a_{i_k}-a_{j_1}-\dots-a_{j_k})=-u(b_{i_1}+\dots+b_{i_k}-b_{j_1}-\dots-b_{j_k}))
\end{eqnarray*}
where the sum is taken over all pairs $(u,v)$ such that $0\leq u<q$, $|v|\leq q$, $GCD(u,v)=1,$ and at least $\frac{n}{10}$ different $i$ satisfy $ub_i \neq va_i$.  This last assumption guarantees that the linear form in the first inequality has at least $0.1n$ nonzero coefficients for every $(u,v)$ we are summing over, so that the Hal\'{a}sz bound above will be sufficiently strong.    

In the final term in the above bound, the inner summand is at most $1$ unless 
\begin{equation*} 
 (a,b)_{i_1}+(a,b)_{i_2}+\dots+(a,b)_{i_k}=(a,b)_{j_1}+(a,b)_{j_2}+\dots+(a,b)_{j_k},
\end{equation*}
an equation which has at most $L_k n^k$ solutions.    

It follows that 
\begin{equation*}
 \P(E_q)=O(q^2 n^{-k-1/2} L_k+n^{-1/2})
\end{equation*}
which by our assumptions on $L_k$ is $O(\frac{q^{1+\frac{1}{2k+1}}}{\sqrt{n}})=O(\frac{q}n^{\alpha-\frac{1}{2}})$  

\textbf{Case 2: $L_j \geq \frac{n}{q^{\frac{2}{2k+1}}}L_{j-1}$} We know that each variable can be involved in at most $O_j( L_{j-1})$ different $j-$tuples which sum to the same value.  It follows that in this case for some absolute constant $C_j$ we can find a collection $S$ of $C_j\frac{n}{q^{\frac{2}{2k+1}}}$ disjoint $j-$tuples, each of which has coefficients summing to the same (fixed and non-random) pair $(c,d)$.  By our assumption on the $b_i$, and $a_i$, we know that either $d$ is positive or $d=0$ and $c$ is positive.  In particular, we know that at least one of $c$ and $d$ is nonzero.   

Define a $j-$tuple $(i_1,\dots i_j)$ to be \textbf{agreeable} if $x_{i_1}=x_{i_2}=\dots=x_{i_j}$.  Note that each tuple has a constant probability $2^{1-j}$ of being agreeable.  Let $S'$ be the collection of tuples in $S$ which are agreeable, and let $B$ be the event that $|S'| \geq 2^{-j} |S|$.  We have 
\begin{equation*}
 \P(E_q) \leq \P(\neg B)+\P(E_q \wedge B)
\end{equation*}
Note that the agreeability of each tuple in $S$ is an independent event due to our assumption that the tuples are disjoint.  It follows by Chernoff's bound that $\P(\neg B)=o(n^{-1/2})$.  We therefore focus on the second term.  

To bound $\P(E_q \wedge B)$, we will expose the variables by first exposing $S'$, then exposing the value of all the variables not involved in a tuple in $S'$.  We will then finally expose the values of the variables in $S'$.   

We have for any tuple that
\begin{equation*}
 \P(\sum_{i=1}^j (a_{i_j}, b_{i_j}) =(c,d) | (i_1, \dots i_j) \textrm{ agreeable })=1/2
\end{equation*}
and the same for $(-c,-d)$.  It follows that, treating the set $S'$ and the value of $x_j$ for variables not in $S'$ as fixed,
\begin{equation*}
 \frac{\sum_{i=1}^n a_i x_i}{\sum_{i=1}^n b_i x_i} \sim \frac{c \sum_{j=1}^{|S'|} y_j +z_1}{d \sum_{j=1}^{|S'|} y_j + z_2},
\end{equation*}
where $z_1$ and $z_2$ are fixed constants and the $y_i$ are independent $\pm 1$ variables.  By paying at most a constant multiplicative factor and an exponentially small additive factor in the probability, we may replace the sum of the $y_j$ by a uniform distribution on $[-2\sqrt{|S'|}, 2\sqrt{|S'|}]$.  We are thus essentially reduced to bounding the probability that $\frac{az+b}{cz+d}$ can be written as a fraction with low numerator and denominator.  We will soon show:
\begin{lemma} \label{lemma:azplusb}
  Let $n \geq 1$ be an integer, and let $a, b, c, d$ be real numbers (which may depend on $n$) such that $ad \neq bc$.  Let $\alpha>0$ be any fixed parameter.  Then for any $1 \leq q \leq n$, there are at most $qn^{\alpha}$ integers $z \in \{1, \dots n\}$ such that 
\begin{equation*}
 h(z):=\frac{az+b}{cz+d}
\end{equation*}
has height at most $q$ (has numerator and denominator at most $q$ in absolute value when written in lowest terms).  
 \end{lemma}

Assuming Lemma \ref{lemma:azplusb} to be true, we know that for fixed $z_1, z_2$ the probability that this fraction can be written as $\frac{u}{v} \neq \frac{c}{d}$ is at most $\frac{q n^{\frac{1}{3k}}}{\sqrt{|S'|}}$.  Taking expectations over all $z_1, z_2, S'$ and using our bounds on $S'$ under the assumption that $B$ holds gives that 
\begin{equation*}
\P(E_q \wedge B) \leq \frac{q^{1+\frac{1}{2k+1}+\frac{1}{3k}}}{\sqrt{n}}+\P(\frac{\sum a_i}{\sum b_i}=\frac{c}{d} \wedge da_i-ca_i \neq 0 \textrm{ for} \frac{n}{10} \textrm{ different }i) 
\end{equation*}
The second term on the right side corresponds to a linear form with $\frac{n}{10}$ nonzero coefficients, so is $O(n^{-1/2})$.  Again the result follows.  

It remains to prove Lemma \ref{lemma:azplusb}.
\subsection{The proof of Lemma \ref{lemma:azplusb}} \footnote{Many of the key ideas in the proof of Lemma \ref{lemma:azplusb} are due to Ernie Croot}

We may without loss of generality assume that $|a| \geq |c|$.  We will further assume without loss of generality that no prime divides all of $a,b,c,d$.    

Let $\Delta=|ad-bc|>0$.  Note that any common divisor of $az+b$ and $cz+d$ is also a common divisor of $|a(cz+d)-c(az+b)|=\Delta$.  Let $\tau(\Delta)$ be the number of divisors of $\Delta$.  We will split into two cases.

\textbf{Case 1:} $\tau(\Delta)<n^{\alpha/2}$.  For $0 \leq i \leq (\alpha+1) \log_2 n$, let $S_i$ denote the set of $z \in \{1, \dots n\}$ such that $|az+b| \in [2^i, 2^{i+1}]$.  It is clear that each $S_i$ lies in the union of two intervals, each of which has size at most $2^{i}$.   For any $z \in S_i$ such that $h(z)$ has height at most $q$, it must be the case that $az+b$ shares a divisor $v$ with $cz+d$ and $\Delta$ such that $v>2^i/q$.   We next claim that for any given $v$, there are not many $v$ for which this can occur, as:
\begin{claim}
 If $v | GCD(az_1+b, cz_1+d)$ and $v | GCD(az_2+b, cz_2+d)$, then $v|z_1-z_2$.  
\end{claim}
\begin{proof}
 Let $p$ be a (fixed) prime dividing $v$, and let $p^m$ be the largest power of $p$ dividing $v$.  If $p$ does not divide $a$, then $p^m$ must divide $z_1-z_2$, since $v|(az_1+b)-(az_2+b)=a(z_1-z_2)$.  Similarly, either $p^m$ divides $z_1-z_2$ or $p$ also divides $c$.  However, $p$ cannot divide both $a$ and $c$, for it would then follow that $p$ also divided $(az_1+b)-az_1=b$ and $d$, violating our assumption that $a,b,c,d$ shared no common factor.  Therefore it must be the case that $p^m|z_1-z_2$.  But this is true for any prime, so we are done.    
\end{proof}
It follows that for a given $v$, there are at most $2^{i+1}/v$ choices of $z$ for which $v$ provides the required cancellation.    Adding up over all $v$, we see that the number of $z \in S_i$ which lead to a height of at most $q$ is at most 
\begin{equation*}
 \sum_{v|\Delta \atop v>2^i/q} \left( \frac{2^i}{v}+1 \right) \leq (q+1)\tau(\Delta) \leq (q+1)n^{\alpha/2}.    
\end{equation*}
Adding up over all $S_i$, we see that the lemma holds in this case.

\textbf{Case 2:} $\tau(\Delta)>n^{\alpha/2}$.  In this case it follows from classical number theoretic bounds on the number of divisors of an integer that $\Delta>2^{\omega(n)}$ for some $\omega(n)$ tending to infinity with $n$.  

Recall that we are assuming that $|a|>|c|$, so in particular $a$ is non-zero.  By paying an (additive) factor of at most $2n^{\alpha/2}$, we may therefore only consider values where $|az+b| \geq n^{\alpha/2}$.

The result will follow immediately if we can show that for any interval of length at most $n^{1-\alpha/2}/q$, there can be at most three such values of $z$ in that interval for which that $h(z)$ has height at most $q$.  Let us then assume to the contrary that there are four values $z_1, z_2, z_3, z_4$ in such an interval for which $h(z)$ has height at most $q$.  

Let $u_i=az_i+b$, let $v_i=cz_i+d$, and let $h(z)=u_i'/v_i'$ be written in lowest terms.  We next make the following claim:

\begin{claim}
 Let $u_i, u_i', v_i, v_i', z_i$ be as above.  Then
\begin{equation*}
 \frac{u_1 u_2 u_3 u_4}{\prod_{1 \leq i <j \leq 4} |z_i-z_j|} \leq lcm(u_1, u_2, u_3, u_4)
\end{equation*}
\end{claim}
\begin{proof}
Since $GCD(u_i, u_j)|[(a z_i+b)-(az_j+b)]=a(z_i-z_j)$, and $GCD(a,b)=1$ by assumption, it follows that $GCD(u_i, u_j)|z_i-z_j$.  We therefore have 
\begin{equation*}
 \frac{u_1 u_2 u_3 u_4}{\prod_{1 \leq i <j \leq 4} |z_i-z_j|} \leq \frac{u_1 u_2 u_3 u_4} {\prod_{1 \leq i <j \leq 4} GCD(u_i, u_j)} \leq lcm(u_1, u_2, u_3, u_4)
\end{equation*}
\end{proof}

Combining this with the observation that $lcm(u_1, u_2, u_3, u_4) | \Delta lcm(u_1', u_2', u_3', u_4')$, we see that 
\begin{equation} \label{eqn:lcmbound}
 q^4 \geq lcm(u_1', u_2', u_3', u_4') \geq \frac{u_1 u_2 u_3 u_4}{\Delta \prod_{1 \leq i <j \leq 4} |z_i-z_j|}
\end{equation}

We now divide into two further cases depending on the size of $a$ relative to $\Delta$.  

\textbf{Case 2a:} $|a| \geq \Delta^{2/5}$.  Then 
\begin{equation*}
|u_1 u_2 u_3 u_4| \geq a^3 \geq q^4 n^4 \Delta \geq q^4 \Delta \prod_{1 \leq i < j \leq 4} |z_i-z_j|,
\end{equation*}
which is a contradiction to \eqref{eqn:lcmbound}.  

\textbf{Case 2b:} $|a|, |c| \leq \Delta^{2/5}$.  Let $M$ be the larger of $|b|$ and $|d|$.  It follows from our bounds on $a$ and $c$ that $M \geq \frac{1}{2} \Delta^{3/5}$.  

Note that for any $z$ in our range we have 
\begin{equation*}
\max\{|az+b|, |cz+d|\} \geq M-\sqrt{n}{\Delta^{2/5}} \geq \frac{M}{2}, 
\end{equation*}
where we are here using our lower bounds on both $M$ and $\Delta$.  It follows that $GCD(az_1+b, cz_1+d) \geq \frac{1}{2} M n^{-1}$, from which we know that 
\begin{equation*}
 |b| \geq \frac{1}{2}M n^{-1}-|az_1| \geq \frac{1}{3} M n^{-1} 
\end{equation*}
and a similar statement for $|d|$.  In other words, both $b$ and $d$ would have to be much larger than both $|a|$ and $|c|$.  This in turn would imply
\begin{equation*}
 |\frac{az+b}{cz+d} - \frac{b}{d}|=|\frac{\Delta z}{bd}| \leq \Delta^{-1/5} n^2 < \frac{1}{2 n^2}. 
\end{equation*}
But an interval of width less than $\frac{1}{n^2}$ can only contain at most one fraction of height less than $n$, since any nonzero difference between two such fractions is at least that large.  We again reach a contradiction.

\subsection{The proof of Lemma \ref{lemma:neighborstruc} for $k>2$}
Let $(v_1, \dots v_k)$ be a neighborly tuple.  We first modify the definition of Commensurability slightly, writing
\begin{equation*}
 Comm^*(a_1, \dots , a_k) = Comm(a_1, \dots a_k) \chi(\prod_{i=1}^k a_i \neq 0).
\end{equation*}
We have by Theorem \ref{thm:quantLO} and the fact the Commensurability is always at most 1 that
\begin{eqnarray*}
 \E_y(Comm^* (v_1^T y, \dots v_k^T y)) &\geq& \E_y (Comm(v_1^T y, \dots v_k^T y)) - \P(\textrm{ some } v_i^Ty=0) \\
&\geq& \E_y(Comm(v_1^T y, \dots v_k^T y)) - k r^{-1/2} \\
&\geq& \frac{1}{12}  r^{-\frac{1}{2}+\frac{5\epsilon}{8}}
\end{eqnarray*}
The advantage to this truncated commensurability is that we have the relationship
\begin{equation*}
 Comm^*(a_1, \dots , a_k) \geq \frac{1}{R} \Leftrightarrow \frac{a_1}{z_1}=\frac{a_2}{z_2} \dots = \frac{a_k}{z_k}
\end{equation*}
for some integers $z_1 \dots z_k$ which are at most $R$ in absolute value.   

As in the $k=2$ case, we have 
\begin{eqnarray} \label{eqn:modcomm} 
\E_y(Comm^*(v_1^T y, \dots , v_k^Ty)) &\leq& (\int_{t=1}^{r^{\frac{1}{2}-\frac{\epsilon}{4}}} \frac{p_{v}(q)}{q^2} dq)+r^{-\frac{1}{2}+\frac{\epsilon}{4}} \\
& & +\P(\frac{v_1^T y}{l_1}=\frac{v_2^T y}{l_2}=\dots=\frac{v_k^T y}{l_k} \textrm{ for a degenerate } l) \nonumber,
\end{eqnarray}
where 
\begin{equation*}
p_{v}(q):= \P( \exists l=(l_1, \dots l_k) :  \textrm{l is non-degenerate } \wedge \frac{a_i^T y}{l_i} \textrm{all equal } \wedge |l_i| \leq q),
\end{equation*}
and a $k-$tuple $(l_1, \dots l_k)$ is degenerate if $(l_i, l_j)$ is degenerate for $(v_i, v_j)$ for every $i$ and $j$.  Note that a given $(v_1, \dots v_k)$ again has (up to multiples) only one degenerate $l$.

It follows from the proof of the $k=2$ case that for any particular $(i,j)$, the contribution to $p_{v}(q)$ from those tuples where $(l_i, l_j)$ is nondegenerate is $O(\frac{q}{r^{1/2-\alpha}})$ for any $\alpha$.  Adding up over all pairs, it follows that $p_v(q)=O(\frac{k^2q}{r^{1/2-\alpha}})$.  As in the $k=2$ case, we now have   
\begin{equation*}
\int_{t=1}^{r^{\frac{1}{2}-\frac{\epsilon}{4}}} \frac{p_{v}(q)}{q^2} dq = O(k^2 r^{-1/2+\alpha} \ln r)=o(r^{-\frac{1}{2}+\frac{5\epsilon}{8}}).
\end{equation*}
 by taking $\alpha$ to be sufficiently small.  Again the contributions from the first two terms on the right hand side of \eqref{eqn:modcomm} are small, so the last term must be large, that is to say
\begin{equation} \label{eqn:kdegen}
 \P(\frac{v_1^T y}{l_1}=\frac{v_2^T y}{l_2}=\dots=\frac{v_k^T y}{l_k}) \geq \frac{1}{14} r^{-\frac{1}{2}+\frac{5\epsilon}{8}}
\end{equation}
Let $d_j=\frac{l_1}{l_j}$, and $S_j$ to be the places where $v_1$ differs from $d_j v_j$.  We can rewrite \eqref{eqn:kdegen} as the system 
\begin{eqnarray*}
 \sum_{i \in S_2} (d_2 v_2(i)-v_1(i))x_i &=& 0 \\
 \sum_{i \in S_3 \backslash S_2} (d_3 v_3(i) -v_1(i)) x_i &=& - \sum_{i \in S_2} (d_3 v_3(i)-v_1(i)) x_i \\
\vdots & & \vdots \\
\sum_{i \in S_k \atop i \notin S_2 \cup \dots \cup S_{k-1}} (d_k v_k(i)-v_1(i)) x_i &=& -\sum_{i \in S_2 \cup \dots \cup S_{k-1}} (d_k v_k(i) - v_1(i)) x_i. 
\end{eqnarray*}
We now successively expose the variables in $S_j \backslash S_2 \cup \dots S_{j-1}$ for each $j$ and examine each equation in turn.

After we expose the variables in $S_2$, the probability that the first equation above holds is at most $|S_2|_1^{-1/2}$ by Theorem \ref{thm:quantLO}.   We now treat the variables in $S_2$ as fixed, meaning that the right hand side of the second equation above is constant, and expose those in $S_3 \backslash S_2$.  For any particular value of the variables in $S_2$, it again follows from Theorem \ref{thm:quantLO} that the probability that the second equation holds is at most $|S_3 \backslash S_2|_1^{-1/2}$.  Continuing onwards through the entire system, we have that the probability that the above system  holds is at most 
\begin{equation*}
 \prod_{j=2}^k |S_j \backslash \bigcup_{i=1}^{j-1} S_i|_1^{-1/2}.  
\end{equation*}
The lemma follows by combining this with \eqref{eqn:kdegen}.

 \subsection{The proof of Lemma \ref{lemma:atypical}}
This proof will follow along very similar lines to that of Lemma \ref{lemma:typical}.  

Again we let $k_0:=\lfloor \log^7 r \rfloor$, and the argument will make use of the following analogue of neighborliness:
\begin{definition}
 A tuple $(v_1, \dots v_k)$ of vectors is \textbf{friendly} if 
\begin{equation*}
 \P(v_1^T y = v_2^T y = \dots = v_k^T y = 0) \geq \frac{1}{3} r^{-1+\epsilon} .  
\end{equation*}
\end{definition}

We again have that there are many friendly $k-$tuples.  
\begin{lemma} \label{lemma:manyfriendly}
 Let $k \leq k_0$.  Under the hypotheses of Lemma \ref{lemma:atypical}, there are at least $m^k(1-\frac{r^{1-\frac{\epsilon}{8}}}{m})$ friendly $k-$tuples whose elements are the transposes of rows in $A$.  
\end{lemma}

We also claim that friendly tuples exhibit a similar structure as neighborly ones:
\begin{lemma} \label{lemma:friendlystruc}
 Let $k \leq k_0$, and let $(v_1, \dots v_k)$ be friendly.  Then there are unique real numbers $d_j$ such that if $S_j$ denotes the places where $v_1$ differs from $d_j v_j$,  then
\begin{equation*}
 \prod_{j=2}^k |S_j \backslash \bigcup_{i=2}^{j-1} S_i|_1 \leq 2 r^{1-2\epsilon}. 
\end{equation*}
\end{lemma}

The proof of Lemma \ref{lemma:atypical} from these two lemmas is exactly the same as that of Lemma \ref{lemma:typical} from Lemmas \ref{lemma:manyneighborly} and \ref{lemma:neighborstruc}.  We will therefore focus on the proofs of the two lemmas, which will again turn out to be similar to the proofs of the corresponding lemmas for friendly tuples.  

\subsection{The proof of Lemma \ref{lemma:manyfriendly}}
We define $Z$ to be those $k-$tuples satisfying
\begin{equation*}
 \P(v_1^T y = v_2^T y = \dots = v_k^T y = 0 \wedge y \textrm{ atypical}) \geq \frac{1}{3} \P(y \textrm{ atypical}).
\end{equation*}
By our assumptions about $A$ every tuple in $Z$ is friendly.  Now consider a tuple $(v_1, \dots v_k; y)$ where the $v_i^T$ are chosen randomly from the rows of $A$ and the $y$ is uniform and random.  We estimate the probability that $y$ is atypical and $v_j^T y=0$ for every $j$ in two different ways.  

\textbf{Method 1:} For any atypical $y$, there are at least $(m-r^{1-\frac{\epsilon}{4}})^k$ choices for the tuple.  It follows that the probability is at least
\begin{equation*}
\frac{(m-r^{1-\frac{\epsilon}{4}})^k}{m^k} \P(y \textrm{ atypical })
\end{equation*}

\textbf{Method 2:} We first choose the $k-$tuple, then bound the probability that $y$ works based on whether or not the tuple is in $Z$.  Doing this gives that the probability is at most 
\begin{equation*}
 \frac{1}{m^k} (|Z|+\frac{1}{3}(m^k-|Z|)) \P(y \textrm { atypical } ).  
\end{equation*}

The result follows by comparing the bounds from the two methods, along with the bound 
\begin{equation*}
(m-r^{1-\frac{\epsilon}{4}})^k \geq m^k(1-\frac{kr^{1-\frac{\epsilon}{4}}}{m}) \geq m^k(1- \frac{r^{1-\frac{\epsilon}{4}+o(1)}}{m}).
\end{equation*}  
\subsection{The proof of Lemma \ref{lemma:friendlystruc}}

We first note that for any $j$, we can view the system $v_1^T y = v_j^T y$ as a single vector equation $\sum_i w_i y_i=0$ in $R^2$, where $w_i=<v_1(i), v_j(i)>$.  Since by assumption this equation is satisfied with probability $\frac{1}{3} r^{-1+\epsilon}$, it follows from the 2-dimensional Theorem \ref{thm:Halasz2d} of Halasz that there must be a 1-dimensional subspace containing all but $O(r^{1-\epsilon})$ of the $w_i$.  In terms of the $v_j$, this says that for each $j$ there is a multiple of $v_j$ differing from $v_1$ in at most $r^{1-\epsilon}$ places.  We will take those multiples to be our $d_j$, and $S_j$ to be the places they differ.  

The relationship $v_1^T y = v_2^T y = \dots =v_k^T y = 0$ is equivalent to the system
\begin{eqnarray*}
 \sum_{i \in S_2} (d_2 v_2(i)-v_1(i))x_i &=& 0 \\
 \sum_{i \in S_3 \backslash S_2} (d_3 v_3(i) -v_1(i)) x_i &=& - \sum_{i \in S_2} (d_3 v_3(i)-v_1(i)) x_i \\
\vdots & & \vdots \\
\sum_{i \in S_k \atop i \notin S_2 \cup \dots \cup S_{k-1}} (d_k(i)-v_1(i)) x_i &=& -\sum_{i \in S_2 \cup \dots \cup S_{k-1}} (d_k(i) - v_1(i)) x_i.\\
\sum_{i \notin S_2 \cup \dots \cup S_{k-1}} v_1(i) x_i &=& -\sum_{i \in S_2 \cup \dots S_{k-1}} v_1(i) x_i, 
\end{eqnarray*}
since the first $k-1$ equations each represent $d_j v_j^T y = v_1^T y$ for some $j$, and the last equation represents $v_1^T y=0$.  As in the proof of Lemma \ref{lemma:neighborstruc}, we expose each variable in $S_2$, then the remainder of $S_3$, then the remainder of $S_4$, and so forth.  After all the variables in $S_2$ through $S_j$ have been exposed, the probability that the remaining variables in $S_{j+1}$ cause the next equation to be satisfied is at by Theorem \ref{thm:quantLO} most   
\begin{equation*}
 |S_{j+1} \backslash \bigcup_{i=2}^{j} S_i|_1^{-1/2}.
\end{equation*}

Since each $S_j$ contains at most $r^{1-\epsilon}$ elements, it follows that there must be at least $r/2$ variables still unexposed by the time we expose $S_k$ and arrive at the last equation.  Therefore the probability this last equation holds is at most $2 r^{-1/2}$, so
\begin{equation*}
 \P(v_1^T y = \dots = v_k^Ty=0) \leq 2r^{-1/2} \prod_{j=2}^k |S_{j+1} \backslash \bigcup_{i=2}^{j} S_i|_1^{-1/2}.
\end{equation*}
The lemma follows.

\section{The proof of Theorem \ref{thm:quadlo}}
We first note that for any $\theta$, 
\begin{equation*}
 \P(x^T A x=L(x)+c) \leq \P(x^T Re(e^{i\theta} A) x=Re(e^{i\theta}(L(x)+c))).
\end{equation*}
Since we can always choose a $\theta$ such that $e^{i \theta} a_{ij}$ has non-zero real part for every $i$ and $j$ for which $a_{ij}$ is nonzero, it suffices to prove the result for the case where the entries of $A$, as well as the coefficients of $L$ and $c$, are real.  We will now assume this to be the case.  

The proof will proceed by contradiction.  Let us assume that for some $\delta$ and all $r_0$ there is an $r>r_0$ and a  matrix $A$ of such that  $\P_A>r^{-1/2+\delta}$ and every row of $A$ has at least $r$ nonzero entries.  

We will use a decoupling argument to relate probabilities involving $\P_A$ to a probability involving $x^TBy$ for a suitable bilinear form $B$.  We will then combine those bounds with Theorem \ref{thm:inversebi} to obtain 
\begin{lemma} \label{lemma:rankone}
 Let $A$ be a matrix satisfying the hypotheses of Theorem \ref{thm:quadlo} such that $\P_A>r^{-1/2+\delta}$.  Then there is a principal minor $A'$ of $A$ of size at least $n-O(\frac{r \log n}{\log^5 r})$ and a rank one matrix $A''$ such that $A'=A''$ everywhere off the main diagonal.   
\end{lemma}
This allows us to essentially reduce to the case where $A$ is rank one.  Let us (for now) assume that this lemma is true.

Without loss of generality we may assume that  $A'$ consists of the first $m$ rows and columns of $A$.  Let $z=(x_1, \dots x_m)^T$.  For any particular values of $x_{m+1}, \dots x_n$, we have the relationship
\begin{equation*}
 x^TAx= z^T A' z + \widetilde{L}(z) + c',
\end{equation*}
where $\widetilde{L}$ and $c'$ are dependent on the exposed variables.  Because $x_i^2=1$ for every $i$, we can further replace $A'$ by $A''$ by changing $c'$.  It follows that 
\begin{equation*}
 \P(x^TAx=L(x)+c) \leq \sup_{\widetilde{L}, c'} \P(z^T A'' z=\widetilde{L}(z)+c') 
\end{equation*}

Since $A''$ has rank one, the quadratic form $z^T A''z$ factors as the square of a linear form. Since we only removed $O(\frac{r \log n}{\log^5 r})$ columns in going from $A$ to $A'$, it follows from our assumptions on $r$ that for sufficiently large $n$ every coefficient of that linear form must be nonzero (as $A''$ still has at least $\frac{r}{2}$ nonzero entries per row).   We will soon show
\begin{lemma} \label{lemma:quadfactor}
 Let $b_1, \dots b_m, c_1, \dots c_m, d$ be real numbers such that all of the $b_i$ are nonzero, and let $\alpha>0$.  Then
\begin{equation*}
 \P((\sum_{i=1}^m b_i x_i)^2=\sum_{i=1}^m c_i x_i +d) = O_\alpha(n^{-1/2+\alpha}),
\end{equation*}
\end{lemma}

Combining Lemma \ref{lemma:quadfactor} with Lemma \ref{lemma:rankone}, we see that if for sufficiently large $n$ we have $\P_A > r^{-1/2+\delta}$, then we also have $\P_A=O(r^{-1/2+\delta/2})$, which is a contradiction.  We now turn to the proofs of the lemmas.
\subsection{The proof of Lemma \ref{lemma:quadfactor}}
We define 
\begin{eqnarray*}
 t_1=\sum_{i=1}^{\lfloor\frac{m}{2}\rfloor} b_i x_i & & s_1=\sum_{i=1}^{\lfloor\frac{m}{2}\rfloor} c_i x_i \\
 t_2=\sum_{i=\lfloor\frac{m}{2}\rfloor+1}^m b_i x_i & & s_2=\sum_{i=\lfloor\frac{m}{2}\rfloor+1}^m c_i x_i
\end{eqnarray*}
In terms of these new variables, we are attempting to show 
\begin{equation} \label{eqn:rankone}
 \P(2t_1t_2+t_1^2+t_2^2=s_1+s_2+d) = O(m^{-1/2+t}).
\end{equation}
The left hand side of \eqref{eqn:rankone} can be thought of as the probability that the point $p$ and the line $l$ are incident, where
\begin{equation*}
 p=(t_2, s_2-t_2^2), \, \, l= \{y=2t_1 x + t_1^2-s_1+d\}.
\end{equation*}
 Note that $p$ and $l$ are independent, as they involve different sets of variables.  We now make use of the following probabilistic variant of the Szemer\'{e}di-Trotter theorem, which is essentially a rescaling of the weighted Szemer\'{e}di-Trotter result of Iosevich, Konyagin, Rudnev, and Ten \cite{IKRT}:
\begin{theorem} \label{thm:weightst}
 Let $(p,l)$ be a point and line independently chosen in $\R^2$.  Let
\begin{equation*}
 q_p:= \sup_{p_0} \P(p=p_0) \, \, \, \, q_l:=\sup_{l_0} \P(l=l_0) 
\end{equation*}
 Then the probability that $p$ and $l$ are incident is bounded by   
\begin{equation*}
 \P(p \in l) =O( (q_p q_l)^{-1/3} + q_p + q_l)
\end{equation*}
\end{theorem}

Since $p$ uniquely determines $t_2$ and $l$ uniquely determines $t_1$, it follows from Theorem \ref{thm:quantLO} that $q_p$ and $q_l$ are at  most $O(m^{-1/2})$.  We are therefore done unless 
\begin{equation} \label{eqn:likelypoint}
 q_p q_l \geq n^{-3/2+\alpha}.
\end{equation}
If \eqref{eqn:likelypoint} holds, it follows that there is some point $p_0$ which is chosen with probability at least $n^{-1+\alpha}$.  From the definition of $p$, we know that there are real numbers $t_0$ and $s_0$ such that
\begin{equation*}
 \P(t_2 = t_0 \wedge s_2=s_0) \geq n^{-1+\alpha}.
\end{equation*}
If follows from the $d=2$ case of Hal\'{a}sz's Theorem \ref{thm:Halasz2d} that the coefficient vectors of $t_2$ and $s_2$ must be close to being multiples of each other, that is to say there is an $|S| \subseteq \{\lfloor\frac{m}{2}\rfloor+1, \dots m\}$ with $|S|>\frac{m}{4}$ and a real number $c_0$ such that $c_j=b_j c_0$ for every $j \in S$.  

We now expose every variable not in $S$.  Once we have done so, we are left with an equation of the form
\begin{equation} \label{eqn:quadfinal}
 (\sum_{j \in S} b_j x_j + d_1)^2=c_0(\sum_{j \in S} b_j x_j) + d_2,
\end{equation}
where $d_1$ and $d_2$ are constants depending on the exposed variables.  For any given $d_1$ and $d_2$, there are at most 2 values of $\sum_{j \in S} b_j x_j$ for which \eqref{eqn:quadfinal} holds.  It therefore follows from Theorem \ref{thm:quantLO} that for any given $d_1$ and $d_2$ the probability that \eqref{eqn:quadfinal} holds is $O(m^{-1/2})$.  Lemma \ref{lemma:quadfactor} follows from taking expectations over all $d_1$ and $d_2$.

\subsection{The proof of Lemma \ref{lemma:rankone}}

We will make use of the following ``decoupling'' lemma (Originally proved in \cite{Sid}) to reduce from the quadratic case to the bilinear one. 
\begin{lemma}
Let $Y$ and $Z$ be independent variables, and let $Z'$ be a disjoint copy of $Z$.  Let $E(Y,Z)$ be an event depending on $Y$ and $Z$.  Then
\begin{equation*}
\P(E(Y,Z))^2 \leq \P(E(Y,Z) \wedge E(Y, Z'))
\end{equation*}
\end{lemma}
In our case this implies that if $X=\{x_1, \dots x_n\}$ is a collection of independent Bernoulli variables partitioned into two disjoint subsets $Y$ and $Z$, then 
\begin{eqnarray*}
 & &\P(x^TAx=L(x)+c)^2 = \P(\sum_{i=1}^n \sum_{j=1}^n a_{ij} x_i x_j=L(x)+c)^2 \\
&\leq& \P(\sum_{i=1}^n \sum_{j=1}^n a_{ij} x_i x_j=L_1(y)+L_2(z)+c \wedge \sum_{i=1}^n \sum_{j=1}^n a_{ij} \widetilde{x_i} \widetilde{x_j}=L_1(y)+L_2(z') + c) \\
&\leq& \P(\sum_{i=1}^n \sum_{j=1}^n a_{ij} x_i x_j -L_1(y)-L_2(z)=\sum_{i=1}^n \sum_{j=1}^n a_{ij} \widetilde{x_i}\widetilde{x_j}-L_1(y)-L_2(z').  
\end{eqnarray*}
where $\widetilde{x_j}=x_j$ if $j \in Y$ and $\widetilde{x_j}=x_j'$ if $j \in Z$, and $L(x)=L_1(y)+L_2(z)$ is the natural decomposition of $L$ into the sum of linear forms on $y$ and $z$.  

Let us further suppose that $|Y|=|Z|$ or $|Y|=|Z|+1$.   All terms only involving variables in $Y$ disappear from the right hand side of this last inequality, and we have 
\begin{equation*}
 \P(x^TAx=L(x)+c)^2 \leq  \P( 2 \sum_{x_i \in Y} \sum_{y_j \in Z} a_{ij} x_i (y_j - y_j') = L_1(y)-L_1(y')+Q(y, y')), 
\end{equation*}
where $Q$ is another quadratic form.  By assumption the left hand side of this equation is at least $r^{-1+2\delta}$, while the right hand side has the form $x^TBy=f(y)$. 

If we further knew that for every $i \in Y$ there were at least $\frac{r}{4}$ different $j \in Z$ such that $A_{ij} \neq 0$, it would follow from Theorem \ref{thm:inversebi} that the matrix $B$ must contain a rank $1$ square submatrix of size $n-O_{\delta}(\frac{r}{\log^6 r})$.   With this observation in mind, we make the following definition:
\begin{definition}
Given a quadratic form $A$, a partition $\{x_1 \dots x_n\}=Y \cup Z$ of the $n$ variables into two disjoint subsets is \textbf{balanced} if for every $x_i \in Y$ there are at least $r$ different $x_j \in Z$ for which $a_{ij} \neq 0$. 
\end{definition}

In terms of our original $A$, we know that for any balanced decomposition of the variables into two equal parts $Y$ and $Z$, the submatrix corresponding to $Y \times Z$ is equal to a rank one matrix except for a few rogue variables.  Our next goal will be to play many such decompositions off of each other.  

Since the reduction to a bilinear form only gives us information about the entries in $Y \times Z$, we will want to choose a collection of balanced decompositions such that many different entries appear in this submatrix for some element of the decomposition.  Motivated by this, we make the following definition:
\begin{definition}
 Let ${\mathcal F}=(Y_1, Z_1) \dots (Y_m, Z_m)$ be a collection of balanced partitions of a set $X=\{x_1, \dots x_n\}$ into  pairs of disjoint subsets of equal size.  We say $\mathcal{F}$ \textbf{shatters} $X$ if for every $i \neq j \neq k \neq l$ there is a  $r=r(i,j,k,l)$ such that $i, j \in Y_r$ and $k, l \in Z_r$.  
\end{definition}
In terms of our decoupling, a shattering collection of partitions means that every pair of off-diagonal entries $a_{ik}$ and $a_{jl}$ will appear simultaneously in the bilinear form for some element of $\mathcal{F}$.  We next show that we don't have to consider too many partitions at once
\begin{lemma}
 If $|X|=2m$, there is an $\mathcal{F}$ of size at most $\lceil \frac{5 \ln n}{\ln (17/16)} \rceil<83 \ln n$ which shatters $X$. 
\end{lemma}
\begin{proof}
 Let a $|\mathcal{F}|$ of size  $\lceil \frac{5 \ln n}{\ln (17/16)} \rceil$ be formed by independently and uniformly choosing  $(Y_s, Z_s)$ from the set of all partitions of $X$ into two parts of equal size.  For any given quadruple $(i,j,k,l)$, the probability that $Y_r$ contains $\{i,j\}$ while $Z_r$ contains $\{i,j\}$ is at least $\frac{1}{17}$, and these events are independent over all $r$.  It therefore follows from the union bound that the probability that $X$ fails to be shattered by this collection is at most
\begin{equation*}
 n^4 (\frac{16}{17})^{|\mathcal{F}|}+\P(\textrm{some } (Y_s, Z_s) \textrm{ is not balanced}).
\end{equation*}
The first term is $O(\frac{1}{n})$ by our choice of $|\mathcal{F}|$.  For the second term, we note that by  standard large deviation techniques the probability that for any given $s$ and $i$ that $x_i \in Y_s$ and there are at most $\frac{r}{4}$ nonzero $a_{ij}$ with $j \in Z_s$ is $O(e^{-r/2})$.  It follows from the union bound and our assumption that on $r$ that the second term is also $o(1)$.  Since a random collection almost surely shatters $X$, there must be at least one shattering collection. 
\end{proof}

We now fix some ${\mathcal F}_0$ which shatters our original set of variables and has size at most $83 \log n$.  For each $r$, we know from Theorem \ref{thm:inversebi} that we can find exceptional sets $Y_s' \subseteq Y_s, Z_s' \subseteq Z_s$ with $|Y_s|, |Z_s|=O(\frac{r}{\log^6 r})$ such that the submatrix of $A$ corresponding to $(Y_s \backslash Y_s') \times (Z_s \backslash Z_s')$ has rank one.   Let
\begin{equation*}
 W=\bigcup_{(Y_s, Z_s) \in {\mathcal F}_0} (Y_s' \cup Z_s'). 
\end{equation*}
Without loss of generality we may assume that $W=\{x_{n-t+1}, \dots x_n\}$. By assumption $t=O(\frac{r \log n}{\log^5 r})$.

For any 4 distinct elements $(i,j,k,l)$ disjoint from $W$, we know from the definition of ${\mathcal F}_0$ and $W$ that for some $s$ the $2 \times 2$ submatrix of $A$ on $\{i,j\} \times \{k, l\}$ appeared in a rank one submatrix of $Y_s \times Z_s$.  It follows that for every set of distinct $(i,j,k,l)$, we have $a_{ik} a_{jl}=a_{jk} a_{il}$.  In particular, for every pair $(j,l)$ with $3 \leq k \neq l \leq n-t$, we have
\begin{equation} \label{eqn:entryfactor} 
a_{jl}=a_{1l} \frac{a_{j2}}{a_{12}}.
\end{equation}
We can therefore take $A'$ to be the principal minor of $A$ on $\{x_3, \dots x_{n-t}\}$, and $A''$ to be the matrix for which the right hand side of \eqref{eqn:entryfactor} also holds for $j=l$.    

\subsection{The proof of Corollary \ref{cor:quadlo2}}
Construct a graph whose vertices are the variables $x_i$, with $x_i \tilde x_j$ for $i \neq j$ iff $a_{ij}$ are nonzero.  By assumption, this graph has average degree at least $m-1$.  It follows that it must contain a subgraph of minimum degree at least $\frac{m-1}{2}$.  In matrix terms, this implies that $A$ contains a principal minor $A'$ such that every row of $A'$ has at least $\frac{m-1}{2}$ nonzero entries.  Without loss of generality we may assume that the minor corresponds to the variables $\widetilde{x}=\{x_1, \dots x_k\}$.  For any fixed value of $x_{k+1} \dots x_n$, the equation $x_T A x=L(x)+c$ becomes 
\begin{equation*}
 \widetilde{x}^T A \widetilde{x}=\widetilde{L}(\widetilde{x})+\widetilde{c}, 
\end{equation*}
an equation which holds with probability $O_{\delta}(m^{-\frac{1}{2}+\delta})$ by Theorem \ref{thm:quadlo}.  The result follows from taking expectations over all values of $x_{k+1} \dots x_n$.

\section{Extensions of the Main Results and Conjectures}

\subsection{Inverse results for more weakly concentrated Bilinear Forms:}  
It is an interesting problem to consider whether there are similar inverse results holding in general for when a bilinear form has polynomially large concentration on one value $P(x^TAy=c) \geq n^{-b}$ for some $b$.  

There are at least two different types of structure that lead to sufficient conditions for this to occur.  One possibility is algebraic: If the coefficient matrix has low rank, then $f(x,y)$ will be equal to $0$ whenever a small number of linear forms is equal to $0$, which may not be too unlikely an event if some of those forms are structured.  For example, if $A$ is chosen to satisfy $a_{ij}=f(i)+g(j)$ (for arbitrary $f$ and $g$), then $x^T A y$ can be expressed as
\begin{equation*}
(x_1+x_2+\dots+x_n)(g(1)y_1+ \dots+g(n) y_n)+(f(1)x_1+\dots+f(n) x_n)(y_1+\dots+y_n)
\end{equation*}
and is $0$ whenever $x_1+\dots+x_n=y_1+\dots+y_n=0$, an event which occurs with probability approximately $\frac{1}{n}$. 

Another possibility is arithmetic: If the entries of the coefficient matrix are  all drawn from a short generalized arithmetic progression of bounded rank, then the output of $x^TAy$ will also lie in such a progression, and will by the pigeonhole  principle take on a single value with polynomial probability.  We conjecture that these two ways, and combinations thereof, are essentially the \textit{only} way a bilinear form can have polynomial concentration, that is to say
  
\begin{conjecture}
 For any $a>0$ there are constants $a_1, a_2, a_3$ and $N_0$ such that for all $n>N_0$ the following holds: If $A$ is an $n \times n$ matrix of nonzero entries such that for $x$ and $y$ uniformly and independently chosen from $\{-1, 1\}^n$, 
\begin{equation*}
 \sup_c \P(x^TAy=c) > n^{-a},
\end{equation*}
then $A$ can be written as $A_1$+$A_2$+$A_3$, where $A_1$ has rank at most $a_1$, the entries of $A_2$ are drawn from a generalized arithmetic progression of rank at most $a_2$ and volume at most $a_3$, and $A_3$ contains at most $\frac{n^2}{\log n}$ nonzero entries.   
\end{conjecture}

\subsection{Higher degrees}
In this section we give several conjectured extentions of the main results to this paper to multilinear and polynomial forms.  We begin with the following (simplified) analogue of Theorem \ref{thm:BiLO}, which can be proved by the same method.  
\begin{theorem} \label{thm:multilinear}
 Let $k$ be a fixed positive integer.  Let $y_1=(x_{1,1}, \dots x_{n,1}), \dots y_k=(x_{1, k}, \dots x_{n, k})$ be $n$ independent vectors uniformly chosen from $\{-1, 1\}^n$, and let 
\begin{equation*}
 A(x):=\sum_{i_1=1}^n \sum_{i_2=1}^n \dots \sum_{i_k=1}^n a_{i_1 i_2 \dots i_k} x_{i_1,1} \dots x_{i_k,k}
\end{equation*}
  be a $k-$multilinear form whose coefficients $a_{i_1 \dots i_k}$ are all nonzero.  Then for any function $f$ of $k-1$ variables, 
\begin{equation} \label{eqn:multilinear}
 \P(A(y_1, \dots y_k)=f(y_2, \dots y_k))=O_k(n^{-1/2})
\end{equation}
\end{theorem}
Again, this is tight for degenerate forms which contain a linear factor.  A natural conjecture would be that non-degenerate forms are significantly less concentrated.  
\begin{conjecture}
 Let $k, A, y,$ and $f$ be as in Theorem \ref{thm:multilinear}.  If there is some $\epsilon>0$ such that
\begin{equation*}
 \P(A(y_1, \dots y_k)=f(y_2, \dots y_k)) \geq n^{-\frac{k}{2}+\epsilon},
\end{equation*}
then there is a partition of $\{y_1, \dots y_k\}$ into disjoint sets $S$ and $T$ and functions $f_1$ and $f_2$ such that $f_1$ depends only the variables in $S$, $f_2$ only on the variables in $T$, and $A$ differs from $f_1 f_2$ in $o(n^2)$ coefficients.
\end{conjecture}
The $k/2$ in this conjecture comes from how $n^{k/2}$ is the typical magnitude of $f$ in the case where the coefficients of $A$ are random (small) integers.

We can also conjecture a polynomial analogue to Theorem \ref{thm:quadlo}, including an analogous inverse theorem to the above multilinear one.  
\begin{conjecture}
 Let $x_1, \dots x_n$ be independent and uniformly chosen from $\{-1, 1\}$, and let 
\begin{equation*}
 f(x_1, \dots x_n) = \sum_{1 \leq i_1 \dots \leq i_k} a_{i_1 \dots i_k} x_{i_1} \dots x_{i_k}
\end{equation*}
 be a degree $k$ homogeneous polynomial with at least $mn^{k-1}$ nonzero coefficients.  Then
\begin{equation*}
 \sup_c \P(f(x_1, \dots x_n)=c)=O(m^{-1/2}).
\end{equation*}
If the above concentration is at least $\Omega_k(m^{-k/2+\epsilon})$, then $f$ differs in only a few coefficients from a polynomial which factors.   
\end{conjecture}
In \cite{CTV}, a proof of the first half of this conjecture was given with $m^{-1/2}$ replaced by $m^{-2^{-\frac{k^2+k}{2}}}$.  For the second half, we do not have a proof of this conjecture even in the case $k=2$.  

\textbf{Acknowledgements:} The author wishes to thank Ernie Croot, Tali Kaufman, Endre Szemer\'{e}di, Prasad Tetali, Van Vu, and Philip Matchett Wood for enlightening conversations and helpful comments on earlier versions of this paper.  He is particularly indebted to Croot for providing much of the argument for Lemma \ref{lemma:azplusb}.

\end{document}